\documentclass[12pt]{amsart}

\usepackage{amsmath}
\usepackage{amsfonts}
\usepackage{amssymb}
\usepackage{amsthm}
\usepackage{color}
\usepackage[utf8]{inputenc}
\usepackage{graphicx}
\usepackage{ytableau}
\usepackage{hyperref}
\usepackage[capitalize,nameinlink,noabbrev,nosort]{cleveref}

\newtheorem{theorem}{Theorem}[section]
\newtheorem{definition}[theorem]{Definition}

\newtheorem{corollary}[theorem]{Corollary}
\newtheorem{proposition}[theorem]{Proposition}
\newtheorem{remark}[theorem]{Remark}
\newtheorem{example}[theorem]{Example}
\newtheorem{question}[theorem]{Question}

\DeclareMathOperator{\core}{core}

\DeclareMathOperator{\Par}{\mathcal{P}}
\DeclareMathOperator{\D}{\mathcal{D}}
\newcommand{\vv}{v}

\newcommand{\numcore}{P}
\newcommand{\ds}{\displaystyle}
\newcommand{\qbinom}[2]{\bgroup\renewcommand*{\arraystretch}{1}\begin{bmatrix} #1 \\ #2\end{bmatrix} \egroup}
\DeclareMathOperator{\Abc}{\mathcal{A}}

\begin{document}
	
	\title
	{Cores of partitions in rectangles}
	
	\author{Arvind Ayyer}
	\address{Arvind Ayyer, Department of Mathematics, Indian Institute of Science, Bangalore - 560012, India}
	\email{arvind@iisc.ac.in}
	
	\author{Shubham Sinha}
	\address{Shubham Sinha, Department of Mathematics, University of California San Diego, 9500 Gilman Drive, La Jolla, CA 92093, USA}
	\email{shs074@ucsd.edu}
	
	\date{\today}

	\begin{abstract}
				For a positive integer $t \geq 2$, the $t$-core of a partition plays an important role in modular representation theory and combinatorics. We initiate the study of $t$-cores of partitions contained in an {$r\times s$} rectangle. Our main results are as follows. We first give a simple formula for the number of partitions in the rectangle {that} are themselves $t$-cores and compute its asymptotics for large $r,s$. We then prove that the number of partitions inside the rectangle whose $t$-cores are a fixed partition $\rho$ is given by a product of binomial coefficients. Finally, we use this formula to compute the distribution of the $t$-core of a uniformly random partition inside the rectangle extending our previous work on all partitions of a fixed integer $n$ (\emph{Ann. Appl. Prob.} 2023). In particular, we show that in the limit as $r,s \to \infty$ maintaining a fixed aspect ratio, we again obtain a Gamma distribution with the same shape parameter $\alpha = (t-1)/2$ and rate parameter $\beta$ that depends on the aspect ratio.
	\end{abstract}
	
	
	\maketitle

\section{Introduction}

The combinatorial study of $t$-cores has its origins in the theory of modular representations (see for instance~\cite{james-kerber-2009}), and has applications in many areas of mathematics. We focus only on combinatorial and probabilistic aspects of $t$-cores in this work. Garvan, Kim and Stanton~\cite{garvan-kim-stanton-1990} stated many combinatorial properties of $t$-cores in their study of cranks. Lulov and Pittel~\cite{lulov-pittel-1999} and independently, we~\cite{ayyer-sinha-2023} studied the sizes of $t$-cores of uniformly random partitions of an integer $n$, and derived the limiting distribution; see \cite{ayyer-sinha-2020} for the extended abstract. Recently, Rostam~\cite{rostam-2021} used  determinantal properties of Plancherel-random partitions to derive a similar result for the sizes of $t$-cores.

We extend our previous work by focusing on the set of partitions contained in an $r \times s$ rectangle, which we denote $\Par_{r,s}$, specifically with respect to $t$-cores and quotients. We will focus on 
enumerative and probabilistic aspects of such partitions. 
We begin with the preliminaries in \cref{sec:prelim}.
In \cref{sec:numcores}, we give an explicit formula for the number of $t$-cores in $\Par_{r,s}$  in \cref{prop:numcores}. We then look at the growth of this number as $s,r$ tend to infinity so that their ratios approach a fixed number. 
In that case, we obtain the leading constant in \cref{thm:asymp-numcores}. 
In \cref{sec:num-fixedcore}, we enumerate partitions in $\Par_{r,s}$ whose $t$-cores are some fixed $t$-core $\rho$
in \cref{thm:partition_fixed_core} and prove a concise formula for the expected size of the $t$-core of a uniformly random partition in $\Par_{r,s}$ in \cref{prop:exp}.
Finally, in \cref{sec:limsize}, we consider the size of the $t$-core of a uniformly random partition in $\Par_{r,s}$ under two limits. We first take $s \to \infty$ with fixed $r$ and give an explicit formula for the probability generating function of the size of the $t$-core.  Then we take the limit again as $r,s$ tend to infinity with $s/r \to \kappa$. We show that this converges in distribution to a Gamma random variable in \cref{thm:gamma}

\section{Preliminaries}
\label{sec:prelim}

Let $r,s \geq 1$ and $t \geq 2$ be fixed integers. A \emph{partition} $\lambda$ is a weakly decreasing sequence of positive integers $\lambda = (\lambda_1,\dots,\lambda_k)$. The \emph{length} of a partition $\lambda$, denoted $\ell(\lambda)$ is defined to be the number of nonzero parts in $\lambda$ and the \emph{size} of $\lambda$ is $|\lambda| = \sum_i \lambda_i$.
Graphically, a partition $\lambda$ can be represented by a \emph{Young diagram}, in which the $i$'th row contains $\lambda_i$ left-justified boxes. We will number the rows from top to bottom, following the English notation. For example, the Young diagram of the partition $\lambda = (5,4,4,1)$, with $|\lambda| = 14$ and $\ell(\lambda) = 4$, is
\begin{equation}
\label{eg-partition}
\ydiagram{5, 4, 4, 1}
\end{equation}
The conjugate of a partition $\lambda$, denoted $\lambda'$, is the partition whose Young diagram is the conjugate of that of $\lambda$. For the example above, it follows from \eqref{eg-partition} that $(5,4,4,1)' = (4, 3, 3, 3, 1)$.
Let $\Par$ be the set of all partitions and $\Par_{r,s}$ be the set of partitions $\lambda$ such that $\lambda_1 \leq s$ and $\ell(\lambda) \leq r$. 
Then $\Par_{r,s}$ is the set of partitions whose Young diagrams fit inside the rectangular partition $(s^r)$. The partition in \eqref{eg-partition} fits in $\Par_{r,s}$ for all $s \geq 5$ and all $r \geq 4$.
It is easy to see that the number of partitions contained in $\Par_{r,s}$ is the binomial coefficient $\binom{r+s}{r}$ and it is also well-known that
\begin{equation}
\label{qbinom}
\sum_{\lambda \in \Par_{r,s}} q^{|\lambda|} = \qbinom{r+s}r_q,
\end{equation}
the $q$-binomial coefficient or Gaussian polynomial~\cite[Proposition 1.7.3]{ec1}.

If the Young diagram of a partition $\mu$ is contained inside that of $\lambda$, the set difference, denoted {$\lambda/\mu$} is called a \emph{skew shape}.
A \emph{rim-hook} is a connected skew shape containing no $2 \times 2$ blocks. {A $t$-rim hook is a rim hook of size $t$.}
The \emph{$t$-core} of a partition $\lambda$, denoted $\core_t(\lambda)$, is obtained by successively removing $t$-rim hooks from the Young diagram of $\lambda$ while maintaining a Young diagram at each stage. Although not obvious, it can be shown that the $t$-core is well-defined, i.e. it does not depend on the order in which the rim-hooks are removed. 
If $\lambda$ has no $t$-rim hooks, $\lambda$ is itself said to be a $t$-core. The partition in \eqref{eg-partition} is a $7$-core. {We say that a partition $\lambda$ is defined to be a $t$-divisible partition if its $t$-core, $\core_t(\lambda)$, is the empty partition. These partitions are in bijection with the set of $t$-quotients; see \cite{ayyer-sinha-2023} for details.} 
For the partition $\lambda$ in \eqref{eg-partition}, $\core_5(\lambda) = (3,1)$, $\lambda$ is a $7$-core and $\lambda$ is $2$-divisible.
See~\cite{james-kerber-2009,loehr-2011,ayyer-sinha-2023} for more properties of $t$-cores.

{There are several equivalent notions that can be used to study $t$-cores of partitions, such as \emph{Maya diagrams}, \emph{star diagrams}, \emph{$\beta$-sets} and \emph{abaci}.
We find the language of abaci most convenient for proving our main results. An \emph{abacus} is an infinite word $w=(w_m)_{m\in\mathbb{Z}}$ in the {alphabet} $\{0,1\}$ such that there exists a positive integer $N$ such that $w_i=0 $ for all $i\ge N$ and $w_i=1$ for all $i \leq -N$. } The \emph{offset} of $w$ (defined in { \cite[Section 2]{olsson} as position}) is given by
\[
d(w)= |\{i \geq 0 \mid w_i = 1\}|-|\{ i < 0 \mid w_i = 0\}|. 
\]
An abacus with zero offset is said to be \emph{balanced}.
Every abacus corresponds to a partition in a natural way. Replacing every $1$ by a $(0,1)$ (up step) and $0$ by a $(1,0)$ (right step) traces out the boundary of a partition.
The partition $\lambda$ in \eqref{eg-partition} corresponds to the abacus
\[
(\dots, 1, 1, 1, 0, 1, 0, 0, 0, 1, 1, 0, 1, 0, 0, 0, \dots).
\]
Let $\Abc$ denote the set of abaci. 
 {There is a well-known bijection~\cite{olsson} (see also \cite[Theorem 11.3]{loehr-2011})} $\psi:\Abc\to\mathbb{Z}\times \Par$, sending $w$ to $(d(w),\lambda)$. In particular, this provides a bijection between $\Par $ and the set of abaci with zero offset given by $\lambda \mapsto \psi^{-1}(0,\lambda)$. 

An abacus $w$ is said to be \emph{justified} {if $\psi(w)=(p, \emptyset)$ for some $p$, which is called the \emph{position of justification} of $w$. Equivalently, $w$ is justified if all $1$s of $w$ occur before all $0$s of $w$ and the position of justification of a justified $w$ is the position of its first $0$.} 
Let $w=\psi^{-1}(0,\lambda)$ for a partition $\lambda$.
{The \emph{$t$-runner abacus} of $w$ is the $t$-tuple $(w^0, \dots, w^{t-1})$, where $w^i =(w_{mt+i})_{m\in \mathbb{Z}}$.}
It can be shown that $\lambda$ is a $t$-core if and only if the abaci $w^i$ are justified for $0\le i\le t-1$ (see {\cite[Chapter 2]{james-kerber-2009}}). {If $w$ is balanced, then the positions of justification $p_i$ of $w^i$ sum to zero.} Moreover, $t$-cores are uniquely determined by {the} sequence of integers $(p_0,\dots,p_{t-1})$ that sum to zero {specifying the positions of justification of the $t$-runner abacus}{~\cite[Bijection 2]{garvan-kim-stanton-1990}}. 
The positions of justification are important because they determine the size of the $t$-core, which is given by 
\begin{equation}
\label{core-size}
|\core_t(\lambda)|=\sum_{i=0}^{t-1}\bigg(\frac{t}{2}p_i^2+ip_i\bigg).
\end{equation}
The \emph{$t$-quotient} of $\lambda$ is the $t$-tuple of partitions $\bar{\mu} = (\mu^0,\mu^1,\dots,\mu^{t-1})$ where $\psi(w^i) \allowbreak =(p_i,\mu^i)$. Furthermore, $\core_t(\lambda)$ is the $t$-core whose positions of justification are $(p_0, \allowbreak \dots,p_{t-1})$.
A fundamental result in this area is the \emph{Littlewood decomposition} or \emph{partition division theorem}~\cite[Theorem 11.22]{loehr-2011}, 
which states that there is a natural bijection between $\Par$ and the set of pairs $(\rho, \bar{\mu})$, where $\rho$ is a $t$-core and $\bar{\mu}$ is the $t$-quotient, such that $|\lambda|=|\rho|+t(|\mu^0|+ \cdots + |\mu^{t-1}|)$.

\section{Number of $t$-cores in a rectangle} 
\label{sec:numcores}

We shall first describe a new version of the partition division theorem for the set $\Par_{r,s}$.

\begin{definition}
	Let $\vv$ be a finite word in the alphabet $\{0,1\}$, i.e. a binary word. We denote $\ell(\vv)$ for the length of the word $\vv$. 
	The \emph{size}, denoted $|\vv| = \sum_i \vv_i$, is the number of $1$'s appearing in $\vv$. A word $\vv$ is said to be \emph{justified} if all the 
	$1$'s in it appear before all the $0$'s.
\end{definition}

Let $\vv$ be a binary word of length $r+s$ whose positions are labelled $0$ through $r+s-1$. For $0 \leq i \leq t-1$, denote 
\begin{equation}
\label{defni}
n_i:=\left\lfloor \frac{r+s+t-1-i}{t} \right\rfloor
\end{equation}
for the number of positions congruent to $i$ modulo $t$.

\begin{proposition}
	\label{prop:bij_tuple_words}
	Fix $t\ge 2$. There is a bijection between the set $\Par_{r,s}$ and a $t$-tuple of finite binary words $(\vv^0,\dots,\vv^{t-1})$ satisfying the following two conditions:
	\begin{itemize}
		\item $\ell(\vv^i)=n_i$, and 
		\item $|\vv^0|+|\vv^1|+\cdots +|\vv^{t-1}|=r$.
	\end{itemize}
	Moreover, a partition $\lambda\in \Par_{r,s}$ is a $t$-core if and  only if each word in $(\vv^0,\dots,\vv^{t-1})$ is justified. 
\end{proposition}

\begin{proof}
	For a partition $\lambda\in \Par_{r,s}$, we consider the segment of the abacus $w=\psi^{-1}(0,\lambda)$ between $-r$ and $s-1$. Let the corresponding word be  $\vv=(w_{-r},\dots,w_{-1},w_0,w_{1},\dots w_{s-1})$.
	Note that $|\vv|=r$ and $\ell(\vv)=r+s$. 
	We may split the above word modulo $t$ to define, for $0\le i\le t-1$,
	\begin{equation}
	\label{def-v^i}
	\vv^i= (w_{-r+i}, w_{-r+i+t}, \dots, w_{-r+i+t(n_i-1)}),
	\end{equation}
	where $n_i$ is given by \eqref{defni}. 
	From the description above, it is clear that $\ell(\vv^i)=n_i$ and $|\vv^0| + \cdots +|\vv^{t-1}|=|\vv|=r$.
	
	Recall that a partition $\lambda$ is a $t$-core if and only if the abaci $w^j=(w_{mt+j})_{m\in \mathbb{Z}}$ are justified. Moreover, each word $\vv^i$ is a segment of the abacus $w^j$ where $j\equiv i-r\pmod t$.
	Since $\lambda \in \Par_{r,s}$, $\vv^i$ is justified if and only if $w^j$ is justified.
\end{proof}

{
	\begin{example}
		Consider the case of $r = s= 2$. Then
		\[
		\Par_{2,2} = \{ \emptyset,\, (1),\, (2),\, (1,1),\, (2,1),\, (2,2) \}.
		\]
		For $t = 3$, we have $n_0 = 2, n_1 = 1, n_2 = 1$. Then the set of $3$-tuples of binary words in natural bijection are
		\[
		\{(10, 1, 0),\, (10, 0, 1),\, (11, 0, 0),\, (00, 1, 1),\, (01, 1, 0),\, (01, 0, 1)\}.
		\]
		Notice that the last two are the only partitions here {that} are not $3$-cores. This corresponds to the fact that the first word in the tuple in both cases is $01$, which is not justified.
	\end{example}
}
{Let $\Par^t_{r,s}$ denote the set of $t$-cores inside the $r\times s $ box.}
\begin{remark}
	\label{rem:posjus}
	{
		Recall the positions of justification of a partition from \cref{sec:prelim}.
		For $\rho\in \Par_{r,s}^t$, the positions of justification $(p_0,\dots,p_{t-1})$ can be computed using {the proof of} \cref{prop:bij_tuple_words} as
		\begin{equation}
		\label{posofjustification}
		p_j= |\vv_{i}|-\bigg\lfloor \frac{r+t-1-i}{t} \bigg \rfloor,  \quad
		j\equiv i-r \pmod t.
		\end{equation}
	}
\end{remark}

A \emph{weak composition} of $r$ into $t$ parts is an ordered tuple of non-negative integers $(a_0,\dots, a_{t-1})$ {that} sum up to $r$. 
{If $p(x)$ is a polynomial, the notation $[x^n]p(x)$ stands for the coefficient of $x^n$ in $p(x)$.}

\begin{proposition}
	\label{prop:numcores}
	{Let $\numcore^t_{r,s}$ denote the number of $t$-cores inside the $r\times s $ box. Then} 
	\begin{equation}
	\label{numcore}
	\numcore^t_{r,s} = 	[q^r]\prod_{i=0}^{t-1} (1+q+ \cdots +q^{n_i}),
	\end{equation}
	where $n_i$ is defined in \eqref{defni}. 
\end{proposition}

\begin{proof}
	Let $\rho \in \Par^t_{r,s}$ correspond to the tuple $(\vv^0,\dots,\vv^{t-1})$ of justified binary words given in \eqref{def-v^i} and let $a_i = |\vv^i|$.
	We thus have a natural bijection between $\Par^t_{r,s}$ and the set of restricted weak compositions $(a_0,\dots, a_{t-1}) $ of $r$ satisfying 
	$0\le a_i\le n_i$. 
	{The equality \eqref{numcore} then follows immediately.}
\end{proof}

\begin{remark}
	From the definition of the $t$-core in \cref{sec:prelim}, it is clear that a partition $\lambda$ is a $t$-core if and only if its conjugate $\lambda'$ is a $t$-core. 
	Therefore, we must have $\numcore^t_{r,s} = \numcore^t_{s,r}$. 
\end{remark}

The $t$-core partition (ex-)conjecture, proved by Granville and Ono \cite{granville-ono-1996}, states that the number of partitions of $n$ {that} are $t$-cores is positive for every nonnegative integer $n$ when $t \geq 4$. {One can ask a similar question for partitions contained in rectangles: When is the number of $t$-cores $\lambda$, with exactly $r$ parts and $\lambda_1=s$, positive?}
The following corollary shows that when $r \geq n_j$, $s \leq n_0$ and the maximum hook length $r+s-1$ is not divisible by $t$, there will always exist a $t$-core $\lambda$ inside $\Par_{r,s}$ with $\ell(\lambda) = r$ and $\lambda_1 = s$.
{This result is of course much more elementary.}

\begin{corollary}
	{Let $\widetilde{\numcore}^t_{r,s}$ denote the number of $t$-cores $\lambda$ such that $\lambda_1=s$ and $\ell(\lambda)=r$. }
	{
	\begin{equation}
	\widetilde{\numcore}^t_{r,s} = 	\begin{cases}
	0 & r+s\equiv 1\bmod t, \\
	\ds [q^{r-n_j}] \prod_{\substack{i=1 \\ i \neq j}}^{t-1} (1+q+q^2+\cdots q^{n_i}) & \text{otherwise},
	\end{cases}
	\end{equation}
}
	where $j\equiv r+s-1 \pmod t$.
\end{corollary}
\begin{proof}
	{Let $\lambda$ be a $t$-core with $\lambda_1=s$ and $\ell(\lambda)=r$, and consider the tuple $(a_0,\dots,a_{t-1})$ defined in the proof of \cref{prop:numcores}. The conditions $\ell(\lambda)=r$ and $\lambda_1=s$ are equivalent to $a_0=0$ and $a_j=n_j$ respectively. Consequently, the computation of $\widetilde{\numcore}^t_{r,s}$ involves selecting $q^0$ and $q^{n_j}$ in the terms corresponding to $i=0$ and $i=j$ in the product from \eqref{numcore} and the other coefficients being arbitrary.
	}
\end{proof}

\begin{corollary}[{\cite[Theorem 2.4(2)]{Zhong-2020}}]
	\label{rem:larges}
	When $s$ is much larger than $r$ ($s\ge rt$ suffices), the number of $t$-cores inside an $r\times s$ rectangle is given by the explicit formula
	\[
	\numcore^t_{r,s} = \binom{r+t-1}{r}.
	\]
\end{corollary}

We note that the bijection used by Zhong~\cite{Zhong-2020} is different from that given in \cref{prop:bij_tuple_words}.

Observe that $\numcore^t_{r,s}$ is a polynomial of degree $(t-1)$ in $r$ when $s$ is large.
The behavior of $\numcore^t_{r,s}$ is more interesting when $r$ and $s$ are comparable. 

\begin{theorem}
	\label{thm:asymp-numcores}
	Let $\kappa\ge1$ be a real number. As $r,s\to \infty$ such that $s/r\to \kappa$, 
	\begin{equation}
	\lim\limits_{r,s\to\infty}\frac{\numcore^t_{r,s}}{r^{t-1}}=\frac{A(t,\kappa)}{t^{t-1}},
	\end{equation} 
	where 
	\begin{equation}
	\label{Atm}
	A(t,\kappa)= \frac{1}{(t-1)!}
	\sum_{j=0}^{\big\lfloor \frac{t}{\kappa+1} \big\rfloor}
	(-1)^j\binom{t}{j} \left( t-(1+\kappa)j \right)^{t-1}. 
	\end{equation}
\end{theorem}

\begin{proof}
	Let $(r_k)$ and $(s_k)$ be a sequence of positive integers such that $\lim_{k\to\infty} r_k= \infty$, $\lim_{k\to\infty} s_k \allowbreak = \infty$ and $\lim_{k\to\infty} s_k/r_k=\kappa$. It is enough to assume that the sequences $(r_k)$ and $(s_k)$ are multiples of $t$. Otherwise, we can use the sandwich theorem to bound the limit of $\numcore^t_{r_k,s_k}/r_k^{t-1}$ using 
	\[
	P_{c_k,d_k}^{t} \le P_{r_k,s_k}^{t} \le P_{c_k+t,d_k+t}^{t} \,,
	\]
	where $c_k=t\lfloor r_k/t \rfloor$ and $d_k=t\lfloor s_k/t \rfloor $.
	
	When $r_k$ and $s_k$ are multiples of $t$, 
	the $n_i$'s defined in \eqref{defni} are independent of $i$.
	We set $n_k=(r_k+s_k)/t$ to be this common value. Then \cref{prop:numcores} implies
	\begin{align*}
	\numcore^t_{r_k,s_k}&=[x^{r_k}]\bigg(\frac{1-x^{n_k+1}}{1-x}\bigg)^t\\
	&=\sum_{j=0}^{t}(-1)^j\binom{t}{j}[x^{r_k}]\frac{x^{(n_k+1)j}}{(1-x)^t}\\
	&= \sum_{j=0}^{\big\lfloor \frac{r_k}{n_k+1} \big\rfloor}(-1)^j\binom{t}{j}\binom{r_k-(n_k+1)j+t-1}{t-1}.
	\end{align*}
	Furthermore, we observe that for $k$ large enough
	\[ 
	\bigg\lfloor \frac{r_k}{n_k+1} \bigg\rfloor
	=\bigg\lfloor \frac{t}{1+s_k/r_k+t/r_k} \bigg\rfloor
	\leq \bigg\lfloor \frac{t}{1+\kappa} \bigg\rfloor.
	\]
	Thus for $k$ large enough,
	\begin{align*}
	\numcore^t_{r_k,s_k}=\sum_{j=0}^{\big\lfloor \frac{t}{\kappa+1} \big\rfloor}(-1)^j\binom{t}{j}\binom{r_k-(n_k+1)j+t-1}{t-1}.
	\end{align*}
	Observe that each of the above summand is a polynomial in $r_k$ and $s_k$ of degree $t-1$. We can pick the leading terms to obtain that for $k$ large enough,
	\begin{align*}
	\numcore^t_{r_k,s_k}& \approx
	\frac{1}{(t-1)!}
	\sum_{j=0}^{\big\lfloor \frac{t}{\kappa+1} \big\rfloor} 
	(-1)^j\binom{t}{j}(r_k-n_k j)^{t-1}
	\\
	&= \frac{1}{(t-1)!}
	\bigg(\frac{r_k}{t}\bigg)^{t-1}
	\sum_{j=0}^{\big\lfloor \frac{t}{\kappa+1} \big\rfloor} (-1)^j\binom{t}{j} 
	\bigg(t- \left( 1 + \frac{s_k}{r_k} \right) j \bigg)^{t-1},
	\end{align*}
	where we have omitted the lower order contribution to the binomial coefficient. This completes the proof.
\end{proof}

It turns out that the limiting formula for the case of $s_k=r_k$, namely  $A(t,1)$, has appeared in the literature in other contexts.

\begin{remark}
	Goddard~\cite{goddard-1945} showed that $A(t,1)$ can be expressed as
	\begin{equation}
	\label{At1}
	A(t,1)=\frac{1}{\pi}\int_{0}^{\infty}\bigg(\frac{2\sin x}{x} \bigg)^tdx. 
	\end{equation}
	We also note in passing that $A(t,1)$ 
	was also studied by Swanepoel \cite[Theorem 1]{swanepoel-2015} where he proved the surprising identity,
	\[
	\sum_{n=1}^{\infty}\frac{\sin^t(nx)}{n^t}=\frac{\pi }{2}\bigg(\frac{x}{2}\bigg)^{t-1} A(t,1)-\frac{x^{t}}{2},
	\]
	for any positive integer $t \ge 2$ and $0\le x\le 2\pi/t$.
\end{remark}

\section{Partitions with fixed core}
\label{sec:num-fixedcore}

Let $\rho \in \Par_{r,s}^t$ be a $t$-core partition. Denote 
\[
\D^t_{r,s}(\rho)=\{\lambda \in \Par_{r,s} \mid \core_t(\lambda)=\rho \}
\]
for partitions contained in the {$r \times s$} rectangle with $t$-core equal to $\rho$.

\begin{theorem}
	\label{thm:partition_fixed_core}
	Let $\rho \in \Par_{r,s}^t$ be a $t$-core. Then we have 
	\[
	\sum_{\lambda \in \D^t_{r,s}(\rho)} q^{|\lambda|} =q^{|\rho|} \prod_{i=0}^{t-1}\qbinom{n_i}{a_i}_{q^t},
	\]
	where $n_i$'s are defined in \eqref{defni}, $a_i = |\vv^i|$ and $\vv^i$'s are defined in \eqref{def-v^i}.
\end{theorem}

\begin{proof}
	For any partition $\lambda\in  \D^t_{r,s}(\rho)$, let $(\vv^0, \dots, \vv^{t-1})$ be the $t$-tuple of finite words defined in \eqref{def-v^i}.
	Moreover, the size is $|\vv^i|=a_i$ and the length is $\ell(\vv^i)=n_i$.
	Let $\mu^i$ be the partition in $\Par_{a_i,n_i-a_i}$ corresponding to $\vv^i$. The partitions $\mu^0,\dots ,\mu^{t-1}$ are the $t$-quotients of $\lambda$ (up to a cyclic shift of indices). The partition division theorem gives us 
	\begin{equation}
	|\lambda|=|\rho|+t(|\mu^0|+\cdots +|\mu^{t-1}|).
	\end{equation}
	The required result follows by noting that 
	\begin{align*}
	\sum_{\mu^i\in\Par_{a_i,n_i-a_i} }^{}q^{t|\mu^i|}=\qbinom{n_i}{a_i}_{q^t}
	\end{align*}
	{using \eqref{qbinom}.}
\end{proof}

\begin{example}
	Let $t=3$ and fix the $3$-core $(2,2,1,1)$ with positions of justification given by $p = (1, 1, -2)$. Using \cref{rem:posjus},
	we get $a = (|\vv^0|, |\vv^1|, |\vv^2|) = (0, 2, 2)$.
	For $r = 4$ and $s = 5$, we have $n_0 = n_1 = n_2 = 3$.
	Then
	\[
	\D^3_{4,5}(2,2,1,1) = \left\{
	\substack{ \ds (2, 2, 1, 1),
		(4, 3, 1, 1),
		(4, 3, 3, 2),\\
		\ds  (5, 2, 1, 1), 
		(5, 3, 3, 1),
		(5, 5, 1, 1), \\
		\ds  (5, 5, 3, 2),
		(5, 5, 4, 1),
		(5, 5, 4, 4)}
	\right\},
	\]
	and 
	\[
	\sum_{\lambda \in \D^3_{4,5}(2, 2, 1, 1)} q^{|\lambda|}
	= q^6(q^6 + q^3 + 1)^2 
	= q^{|(2, 2, 1, 1)|} \qbinom 30_{q^3} \qbinom 32_{q^3} \qbinom 32_{q^3},
	\]
	as expected.
\end{example}

Setting $q=1$ in \cref{thm:partition_fixed_core} leads to:

\begin{corollary}
	\label{cor:partition_fixed_core}
	Let $\rho \in \Par_{r,s}^t$ be a $t$-core. 
	The number of partitions $\lambda$ in $\D^t_{r,s}(\rho)$ is
	$\ds \prod_{i=0}^{t-1} \binom{n_i}{a_i}$.
\end{corollary}

\noindent
Recall that the \emph{hypergeometric distribution} {with parameters $(N,K,n)$} describes the probability of $k$ successes in $n$ draws (without replacement) from a population of size $N$ of which $K$ are successes. 
{
	If $X$ is the hypergeometric random variable, the probability mass function is given by 
	\[
	\mathbb{P}(X = k) = \frac{\binom{K}{k} \binom{N-K}{n-k}}{\binom{N}{n}}.
	\]
}
The $t$-dimensional \emph{multivariate hypergeometric distribution}~\cite[Section 7.3.3]{mccullagh-nelder-1989} is a natural generalisation, 
which we will now describe using the above notation for convenience.
Consider a population of size $r+s$ {that} has $t$ kinds of objects whose numbers are described by a tuple $(n_0,\dots,n_{t-1})$ of nonnegative integers summing to $r+s$. 
Then this distribution describes the probability of observing $a_i$ individuals of type $i$, where $0 \leq a_i \leq n_i$,  in $r$ draws without replacement.
	Let $A = (A_0,\dots,A_{t-1})$ be a $t$-dimensional multivariate hypergeometric random vector with parameters $(n_0,\dots,n_{t-1})$ (summing to $r+s$) and $r$. Then, its probability mass function is given by
	\begin{equation}
	\label{pmf-mult-hyper}
	\mathbb{P} \big((A_0,\dots,A_{t-1}) =  (a_0,\dots,a_{t-1}) \big) = 
	\frac{\ds \prod_{i=0}^{t-1} \binom{n_i}{a_i}}{\ds\binom{r+s}{r}}.
	\end{equation}
The fact that this is a probability distribution follows from the multinomial theorem. {Note that $A_i$ has hypergeometric distribution with parameters $(r+s,n_i,r)$.}
We will need the following {standard} facts; see~\cite{johnson-et-al-1997}, for example. The expectation and variance of $A_i$ are given by
\begin{equation}
\label{expvar}
\begin{split}
\mathbb{E}(A_i) &= \frac{r\  n_i}{r+s}, \\
\text{Var}(A_i) &= \frac{r\  s\  n_i}{(r+s)(r+s-1)} \left(1 - 
\frac{n_i}{r+s} \right).
\end{split}
\end{equation}

From \cref{cor:partition_fixed_core}, it is clear that the {random vector $(A_0, \dots, A_{t-1})$ corresponding to the core, 
	$\core_t(\lambda)$, of} a uniformly random partition $\lambda$ in $\Par_{r,s}$ has the $t$-dimensional multivariate hypergeometric distribution with parameters $(n_0,\dots,n_{t-1})$ that sum to $r+s$, and $r$.

\begin{proposition}
	\label{prop:exp}
	The expectation of the size of the $t$-core of a uniformly random partition in $\Par_{r,s}$ equals
	{\begin{equation}\label{eq:expected-tcore-explicit-formula}
		\frac{r s}{(r+s)(r+s-1)}
		\bigg( \binom{r+s \bmod t}{2}+\bigg\lfloor\frac{r+s}{t} \bigg\rfloor \binom{t}{2} \bigg).
		\end{equation}}
\end{proposition}

\begin{proof}
	The expected size of the core of a uniformly random partition $\lambda \in \Par_{r,s}$ is, by definition,
	\begin{equation*}
	\mathbb{E}(|\core_t(\lambda)|) =  \binom{r+s}{r}^{-1}\sum_{\rho \in \Par^t_{r,s}}	|\rho| \; |\D^t_{r,s}(\rho)|.
	\end{equation*}
	Let the $t$-core $\rho$ correspond to the tuple $(a_0, \dots, a_{t-1})$ according to \cref{thm:partition_fixed_core}.
	The probability of observing a fixed core $\rho$ is then given by \eqref{pmf-mult-hyper} using \cref{cor:partition_fixed_core}.
	
	{
		We first focus on the case when both $r$ and $s$ are divisible by $t$.
		Then it is easy to see from \eqref{defni} that $n_i = (r+s)/t$ for all $i$
		and from \cref{rem:posjus} that the position of justification is $p_i = a_i - r/t$ for all $i$. 
		Therefore, we may express the size of the core  using \eqref{core-size} as
		\[
		|\rho|=\sum_{i=0}^{t-1} \bigg( \frac{t}{2} \left( a_i - \frac{r}{t} \right)^2
		+ i \left(a_i-\frac{r}{t} \right) \bigg).
		\]
		In terms of the $t$-dimensional hypergeometric distribution, the expected size of the $t$-core can be written as
		\[
		\mathbb{E}(|\core_t(\lambda)|)=	
		\mathbb{E} \left(
		\sum_{i=0}^{t-1} \bigg( \frac{t}{2} \left( A_i - \frac{r}{t} \right)^2
		+ i \left(A_i-\frac{r}{t} \right) \bigg)
		\right).
		\]
		Since $\mathbb{E}(A_i)=r/t$ from \eqref{expvar}, the second term cancels and we obtain
	}
	\begin{align*}
	\mathbb{E}(|\core_t(\lambda)|)=	\frac{t}{2}\sum_{i=0}^{t-1}\text{Var}(A_i) = \frac{(t-1)}{2}\cdot \frac{r\  s\ }{(r+s-1)}.
	\end{align*}
	
	When $r$ and $s$ are not multiples of $t$, write $r=kt+\epsilon$ and $s=\ell t+\theta$ for $0\le \epsilon,\theta<t$. 
	Using \eqref{core-size} and \eqref{posofjustification}, the size of the $t$-core is given by
	\begin{align*}
	|\rho|=&\sum_{i=0}^{\epsilon-1} \bigg( \frac{t}{2} \left( a_i - (k+1) \right)^2
	+ (t-\epsilon+i) \left(a_i-(k+1) \right) \bigg)\\
	&+\sum_{i=\epsilon}^{t-1} \bigg( \frac{t}{2} \left( a_i - k \right)^2
	+ (i-\epsilon) \left(a_i-k \right) \bigg).
	\end{align*}
	We can express the expectation $\mathbb{E}(|\rho|)$ in terms of the expectation and the variance of the random variables $A_0, \dots, A_{t-1}$. After simplifying, we obtain 
	\begin{align*}
	\mathbb{E}(|\rho|)&= \sum_{i=0}^{\epsilon-1}\bigg( \frac{t}{2} 
	\left( \mathbb{E}(A_i) - (k+1) \right)^2
	+ (t-\epsilon+i) \left(\mathbb{E}(A_i)-(k+1) \right) \bigg)\\
	+& \sum_{i=\epsilon}^{t-1}\bigg( \frac{t}{2} 
	\left( \mathbb{E}(A_i) - k \right)^2
	+ (i-\epsilon) \left(\mathbb{E}(A_i)-k \right) \bigg)
	+ \frac{t}{2} \sum_{i=0}^{t-1}\text{Var}(A_i).
	\end{align*}
	Now, let {$\psi=(\epsilon+\theta) \bmod t$} and
	\[
	\delta = \begin{cases}
	0 &  \epsilon+\theta<t,\\
	1 & \epsilon+ \theta \ge t.
	\end{cases}
	\] 
	Then 
	\[
	n_i=\begin{cases}
	k+\ell+\delta+1 & 0\le i\le \psi-1,\\
	k+\ell+\delta & \psi\le i\le t-1.
	\end{cases} 
	\]
	We obtained an expression for $\mathbb{E}(A_i)$ and $\text{Var}(A_i)$ by substituting the above values of $n_i$ in \eqref{expvar}. The required expression \eqref{eq:expected-tcore-explicit-formula} is obtained by carefully simplifying the above expression in the two cases: $\epsilon+\theta<t$ and $\epsilon+\theta\ge t$.	The details are left to the interested reader.
\end{proof}

\section{Limiting distribution of the size of a $t$-core}
\label{sec:limsize}

We first look at the size of the $t$-core of a uniformly random partition in $\Par_{r,s}$ in the limit of $s \to \infty$ for a fixed $r$. By \cref{rem:larges}, the number of $t$-cores in this case is finite and given by $\binom{r+t-1}{r}$. 

Recall that the \emph{multinomial distribution} is the generalization of the binomial distribution where there are $t$ types of objects with probabilities $q_0, \dots, q_{t-1}$ each of being selected. 
Let $N_i$ be the number of objects of type $i$ selected in $n$ draws with replacement for $0 \leq i \leq t-1$.
Then the probability of seeing the tuple $(n_0, \dots, n_{t-1})$ in $n$ draws is given by the probability mass function
\[
\mathbb{P} \big((N_0, \dots, N_{t-1}) =  (n_0, \dots, n_{t-1}) \big) = 
\binom{n}{n_0, \dots, n_{t-1}} \prod_{i=0}^{t-1} q_i^{n_i}.
\]
If all $q_i$'s are equal to $1/t$, we say that the multinomial distribution is \emph{central}.

For a discrete random variable $X$ taking values in $\mathbb{N}$, the \emph{probability generating function} is given by
\[
\phi(z) = \sum_{n=0}^\infty \mathbb{P}(X = n) z^n.
\]
By taking the limit of the multivariate hypergeometric distribution as $s \to \infty$, we obtain:

\begin{proposition}
	\label{prop:prob-gen-fn}
	The probability generating function of the size of the $t$-core of a uniformly random partition $\lambda\in \Par_{r,s}$ as $s\to \infty$
	is given by
	\begin{align*}
	\phi(z) = \sum_{\rho \in \Par_{r,rt}^t}
	\frac{z^{|\rho|}}{t^r} \binom{r}{a_0,\dots,a_{t-1}},
	\end{align*}
	{where $\rho \in \Par_{r,rt}^t$ corresponds to $(v^0,\dots,v^{t-1})$ 
	as in \cref{prop:bij_tuple_words}  and $a_i=|v^i|$ for all $i$}.
	
\end{proposition}

\begin{proof}
	It can be easily seen {from \cref{prop:bij_tuple_words}} that the probability of $\rho$ being the $t$-core of a uniformly random partition in $\Par_{r,s}$ in the limit as $s \to \infty$ is given by
	\[
	\frac{1}{t^r} \binom{r}{a_0,\dots,a_{t-1}},
	\]
	which describes the central multinomial distribution with parameters $r$ and $t$ and all probabilities equal to $1/t$. 
	This phenomenon is the analog of the hypergeometric distribution tending to the binomial distribution as the population size tends to infinity.
	From this, the formula for the probability generating function easily follows.
\end{proof}

We now consider the rectangular analog of the question we considered in \cite{ayyer-sinha-2023}. Recall that the Gamma distribution is a continuous distribution on the positive real line depending on a \emph{shape} parameter $\alpha > 0$ and a \emph{rate} parameter $\beta$ with density given by
\[
\frac{\beta^\alpha}{\Gamma(\alpha)} x^{\alpha - 1} e^{-\beta x}.
\]
Our main result is as follows.

\begin{theorem}
	\label{thm:gamma}
	As $r,s \to \infty$ so that $s/r \to \kappa$, the random variable given by $t |\core_t(\lambda)|/r$ for a uniformly random partition $\lambda$ in $\Par_{r,s}$ converges to a Gamma-distributed random variable with shape parameter $\alpha = (t-1)/2$ and rate parameter {$\beta = (1+\kappa)/(\kappa t)$}. 
\end{theorem}

\begin{center}
	\begin{figure}[h!]
		\hspace{4cm}
		\includegraphics[scale=0.7]{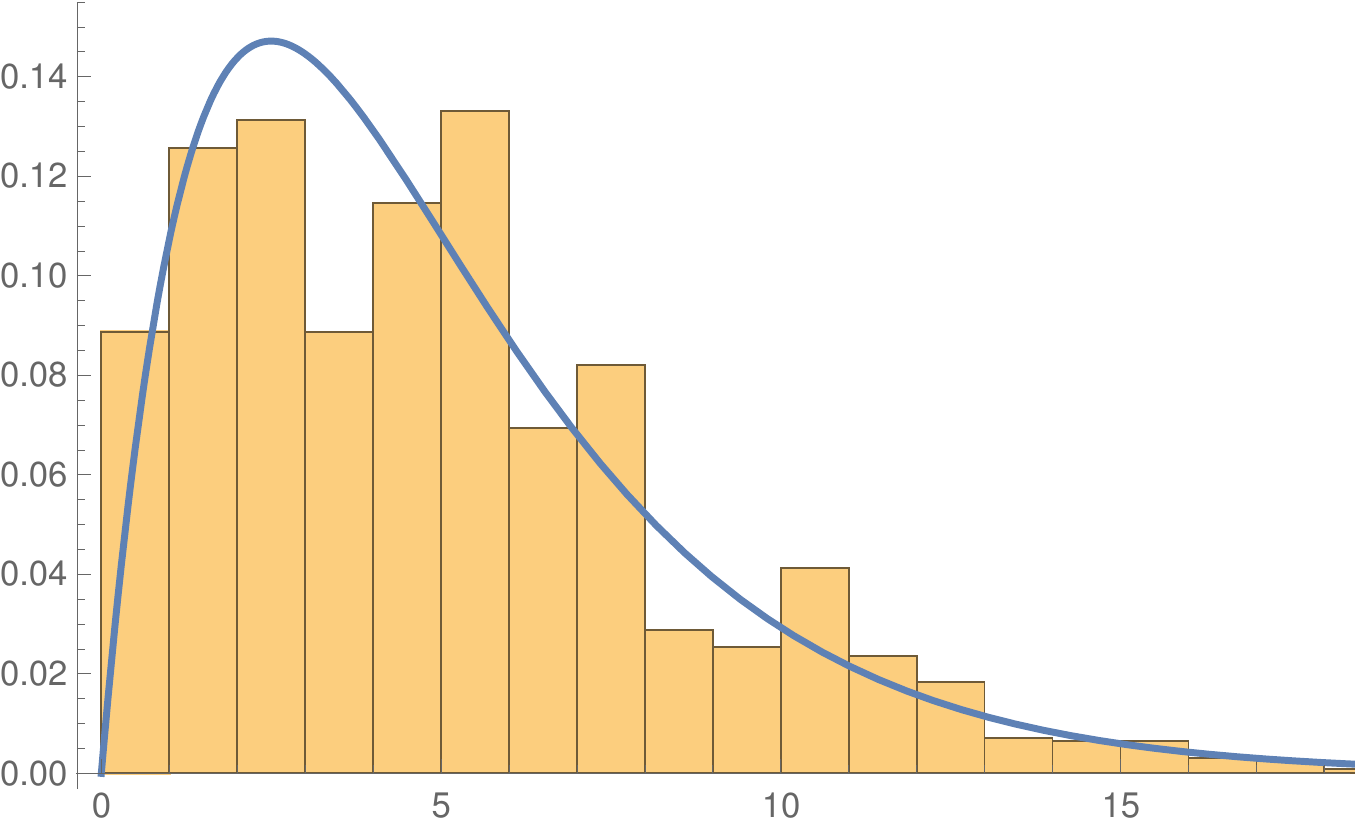}
		\caption{The histogram of $(5/12)$'th the size of the $5$-core for partitions inside $\Par_{12,12}$ and the density for the Gamma distribution with $\alpha = 2$ and $\beta = 2/5$. }
		\label{fig:5core-eg}
	\end{figure}
\end{center}

An illustration of this result is given in \cref{fig:5core-eg}.
To prove the above, we will need a standard convergence result for multivariate hypergeometric random variables. 

\begin{theorem}[{\cite[Chapter 39]{johnson-et-al-1997}}]
	\label{thm:hypergom-conv}
	Suppose $r+s$ is divisible by $t$.
	Let $(A_0,\dots,A_{t-1})$ have the multivariate hypergeometric distribution with parameters $(\underbrace{K,\dots, K}_t)$ and $r$, where $K = (r+s)/t$. Let $Z_i = (A_i - r/t)/ \sqrt{r/t}$ for $1 \leq i \leq t$. Then, as $r, s \to \infty$ with $s/r \to \kappa$, $(Z_0,\dots,Z_{t-1})$ converges to 
	a multivariate normal distribution with density
	\begin{equation}
	\label{dens-normal1}
	f(z_0,\dots,z_{t-1}) = \delta_{z_0 + \cdots + z_{t-1},0}
	\sqrt{\frac{t}{(2 \pi \sigma^2)^{t-1}}}
	\exp \left( - \frac{z_0^2 + \cdots + z_{t-1}^2}{2 \sigma^2} \right),
	\end{equation}
	where $\sigma^2 = \kappa/(\kappa+1)$ and $\delta$ is the Kronecker delta function.
\end{theorem}

We now use this to prove \cref{thm:gamma}.
The key tool in the proof is Stirling's formula approximating the factorial. 

\begin{proof}[Proof of \cref{thm:gamma}]
	By a sandwiching argument similar to the one used in the proof of \cref{thm:asymp-numcores}, it will suffice to consider sequences where {$r$ and $s$ are} divisible by $t$. {Using \eqref{core-size} and \eqref{posofjustification},
		\begin{equation}
		\begin{split}
		\frac{t}{r} |\core_t(\lambda)| 
		= &\frac{t}{r} \sum_{i=0}^{t-1}
		\left( \frac{t}{2} \frac{Z_i^2 r}{t} + i Z_i \sqrt{\frac{r}{t}} \right) \\
		= & \frac{t}{2} \sum_{i=0}^{t-1} Z_i^2 + \sqrt{\frac{t}{r}} \sum_{i=0}^{t-1} i Z_i,
		\end{split}
		\end{equation}
		where $Z_i=(a_i-r/t)/\sqrt{r/t}$ for $0\le i\le t-1$.
	}
	
	Now, as $r \to \infty$, the second term vanishes. Therefore, we end up with something that looks like a chi-squared random variable with $t$ degrees of freedom, except for the fact that the $Z_i$'s sum to $0$. 
	Applying \cref{thm:hypergom-conv}, one can consider the density in \eqref{dens-normal1} as a function of the variables $z_1, \dots, z_{t-1}$ with $z_0 = - (z_1 + \dots + z_{t-1})$. After some manipulations, one finds that the new density $\hat{f}$ can be written as that of a multivariate normal, i.e.
	\begin{equation}
	\hat{f}(z_1,\dots,z_{t-1}) = \frac{1}{\sqrt{(2\pi)^{t-1} \det \Sigma}}
	\exp \left( - \frac{1}{2} \underline{z}^T \Sigma^{-1} \underline{z} \right),
	\end{equation}
	where $\underline{z}^T = (z_1, \dots, z_{t-1})$. Here $\Sigma$ is the covariance matrix whose inverse is given by
	\[
	\Sigma^{-1} = \frac{1}{\sigma^2} \left( I_{t-1} + J_{t-1} \right),
	\]
	where $I_n$ and $J_n$ are the identity matrix and matrix of all ones of size $n$ respectively.
	One can easily show that $\Sigma$ is $\sigma^2/t$ times a matrix whose diagonal entries are all equal to $t-1$ and off-diagonal entries are all equal to $-1$.
	
	To determine the limiting distribution of the size of the $t$-core, we now calculate the moment generating function of $t/2 |Z|^2$. But this is easy to do since
	\begin{equation*}
	\begin{split}
	\mathbb{E}(e^{u t |Z|^2/2}) = 
	\sqrt{\frac{t}{(2 \pi \sigma^2)^{t-1}}}
	& \int dz_0 \cdots \int dz_{t-1} \; 
	\delta_{z_0 + \cdots + z_{t-1}, 0} \\
	& \times \exp \left(\left(  \frac{ut}{2} - \frac{1}{2 \sigma^2} \right) |z|^2 \right),
	\end{split}
	\end{equation*}
	which only modifies the scaling factor. A short computation shows that
	\[
	\mathbb{E}(e^{u t |Z|^2/2}) = (1 - u t \sigma^2)^{-(t-1)/2},
	\]
	which is precisely the moment generating function of the desired gamma distribution.
\end{proof}

Using the well-known fact that the expectation of a Gamma$(\alpha, \beta)$ random variable is $\alpha/\beta$, we get a good sanity check for \cref{thm:gamma} by comparing with the leading term for \cref{prop:exp}.

\begin{corollary}
	\label{cor:exp}
	In the limit as $r, s \to \infty$ with $s/r \to \kappa$,  
	\[
	\mathbb{E} \left( \frac tr |\core_t(\lambda)| \right) = 
	\binom{t}{2} \frac{\kappa}{1+\kappa}.
	\]
\end{corollary}

\begin{question}
	A natural question would be to generalize our results in \cref{prop:prob-gen-fn} and \cref{thm:gamma} when the weight of a partition $\lambda$ is $q^{|\lambda|}$.  This would correspond to $q$-analogs of the multinomial and multivariate hypergeometric distributions, which have not been studied as far as we know.
\end{question}

	\subsection*{Acknowledgements}
{We thank the referees for an exceptionally careful reading of the manuscript and for many useful comments.	
	The first author (AA) acknowledges support from SERB Core grant CRG/ 2021/001592 as well as the DST FIST program - 2021 [TPN - 700661].}

	\bibliography{../cores.bib}
	\bibliographystyle{alpha}

\end{document}